\documentclass[11pt]{article} 
\usepackage{amssymb, amsmath, latexsym}
\usepackage{booktabs, url}
\newtheorem{defin}{}
\newtheorem{saetze}[defin]{}
\newtheorem{conjec}[defin]{}
\newtheorem{lemmas}[defin]{}
\newtheorem{folger}[defin]{}
\newtheorem{bemerk}[defin]{}

\newenvironment{theorem}  {\begin{saetze}\it {\bf Theorem:}}{\end{saetze}}

\newenvironment{lemma}    {\begin{lemmas}\it {\bf Lemma:}}{\end{lemmas}}

\newenvironment{remark}   {\begin{bemerk}\it {\bf Remark:}}{\end{bemerk}}
\newenvironment{proof}    {{\it Proof}:}{{\hfill \fillbox \bigskip}}

\newcommand{\fillbox}{\mbox{$\bullet$}}
\newcommand{\ra}{\rightarrow}

\newcommand{\ms}{\mapsto}
\newcommand{\ol}{\overline}

\newcommand{\g}{\mathfrak{g}}

\newcommand{\N}{\mathbb N}

\newcommand{\Z}{\mathbb Z}

\newcommand{\Hom}{\operatorname{Hom}}

\newcommand{\im}{\operatorname{im}}
\newcommand{\emb}{\hookrightarrow}
\newcommand{\pro}{\rightarrow \!\!\!\! \rightarrow}
\newcommand{\GL}{\operatorname{GL}}

\newenvironment{items}{\begin{list}{$\alph{item})$}
{\labelwidth18pt \leftmargin18pt \topsep3pt \itemsep1pt \parsep0pt}}
{\end{list}}

\begin{document}

\title{Group extensions with special properties}
\author{Andreas Distler and Bettina Eick}
\date{\today}
\maketitle

\begin{abstract}
Given a group $G$ and a $G$-module $A$, we show how to determine up to 
isomorphism the extensions $E$ of $A$ by $G$ so that $A$ embeds as 
smallest non-trivial term of the derived series or of the lower central 
series into $E$. 
\end{abstract}

\section{Introduction}

Given a group $G$ and a $G$-module $A$ an \emph{extension of $A$ by $G$} is 
a group $E$ that satisfies a short exact sequence $A \emb E \pro G$. The 
determination of extensions facilitates the construction of new groups from 
given ones and is an important tool in group theory. 
The isomorphism problem for group extensions asks to determine extensions
{\em up to isomorphism}: given $G$ and $A$, the aim is to determine a
complete and irredundant set of isomorphism type representatives of
extensions of $A$ by $G$. 

We consider two particular types of extensions.
Let $E$ be an extension of $A$ by $G$ and denote by $\ol{A}$ the image of 
the embedding $A \emb E$. We call $E$ a {\em lower central series extension} 
if $E$ is nilpotent and $\ol{A}$ coincides with the smallest
non-trivial term of the lower central series of $E$.
We call $E$ a {\em derived series extension} if $\ol{A}$ 
coincides with the smallest non-trivial term of the derived series of $E$. 

As central part of this paper we present criteria which allow us to identify 
those elements in the cohomology group $H^2(G, A)$ that correspond to the 
desired extensions. For this purpose we use the Schur multiplier of a pair
of groups and a map that can be considered as generalisation of the 
projection map in the universal coefficients theorem. Based on this, we
then describe practical methods to determine up to isomorphism all lower 
central series extensions respectively all derived series extensions of 
$A$ by $G$. 

We apply our algorithms in the construction of groups with large derived
length and small composition length. In particular, we construct two new 
examples of groups of derived length $10$ and composition length $24$. It
is conjectured that $24$ is the minimal possible composition length for
a group of derived length $10$, see \cite{Gla05}.

\section{Extensions}

In this section we briefly recall some notation. For a more detailed
introduction into the theory of group extensions we refer to 
\cite[Chapter 11]{Rob82}. Throughout this section, let $G$ be an arbitrary 
group and let $A$ be an abelian group equipped with a $G$-module structure. 
We write $G$ as multiplicative group and $A$ as additive group. We denote 
the action of $G$ on $A$ by $a^g$ for $a \in A$ and $g \in G$ and we write 
$\ol{g} : A \ra A : a \ms a^g$. 

Let $E = \{(g,a) \mid g \in G, a \in A\}$. Then $\delta \in Z^2(G,A)$
defines a group structure on $E$ via 
\[ (g,a) (g',a') = (gg', a^{g'} + a' + \delta(g,g')). \]
The module $A$ embeds into $E$ via $A \emb E : a \ms (1,a)$ and $E$ projects 
onto $G$ via $E \pro G : (g,a) \ms g$. Thus $E$ is an extension of $A$ by 
$G$. It is well-known that each
extension of $A$ by $G$ is isomorphic to an extension obtained by some
$\delta \in Z^2(G,A)$. We usually identify $A$ with its image $\ol{A}
= \{ (1,a) \mid a \in A\}$.

Let $E_1$ and $E_2$ be two extensions of $A$ by $G$. We say that $E_1$
is \emph{strongly isomorphic} to $E_2$ if there exists an isomorphism
$\iota : E_1 \ra E_2$ with $A^\iota = A$. Further, $E_1$ is
\emph{equivalent} to  $E_2$ if there exists an isomorphism $\iota :
E_1 \ra E_2$ with $A^\iota = A$ so that $\iota$ induces the
identity on $A$ and on $E_1/A \cong G \cong  E_2/A$.

The group of compatible pairs $Comp(G,A)$ is defined as subgroup of
the direct product $Aut(G)\times Aut(A)$ via
\[ Comp(G,A) = \left\{ (\eta, \nu) \in Aut(G) \times Aut(A) \mid
     \ol{g^\eta} = \nu^{-1} \ol{g} \nu \mbox{ for all } g \in G \right\}. \] 
If the action of $G$ on $A$ is trivial then $Comp(G,A)$ equals
$Aut(G)\times Aut(A)$. An action of $Comp(G,A)$ on $Z^2(G,A)$ is
given via
\[ \delta^{(\eta, \nu)} : G \times G \ra A : (g,h) \ms 
   \left(\delta(g^{\eta^{-1}}, h^{\eta^{-1}})\right)^\nu.\]

For $\delta \in Z^2(G, A)$ we denote $[\delta] = \delta + B^2(G,A) \in 
H^2(G,A)$. The subgroup $B^2(G,A)$ of $Z^2(G,A)$ is setwise invariant 
under the action of $Comp(G,A)$ and hence $Comp(G,A)$ acts on $H^2(G,A)$
via $[\delta]^{(\eta, \nu)} = [\delta^{(\eta, \nu)}]$.

\begin{theorem} \label{sisom} {\rm (Robinson \cite{Rob81})} \\
Let $\delta, \delta_1, \delta_2 \in Z^2(G,A)$ and denote their corresponding 
extensions by $E, E_1$ and $E_2$. Write $Aut_A(E) = \{ \alpha \in Aut(E) 
\mid A^\alpha = A \}$. 
\begin{items}
\item[\rm (1)]
Then $E_1$ is strongly isomorphic to $E_2$ if and only if there exists 
$(\eta, \nu) \in Comp(G,A)$ with $[\delta_1]^{(\eta, \nu)} = [\delta_2]$.
\item[\rm (2)]
The homomorphism $\varphi : Aut_A(E) \ra Aut(G) \times Aut(A) : \alpha \ms 
\alpha_{E/A} \times \alpha_A$ satisfies $ker(\varphi) \cong Z^1(G,A)$ and 
$im(\varphi) = Stab_{Comp(G,A)}([\delta])$.
\end{items}
\end{theorem}

\section{Cohomology and Schur multipliers}
\label{pairs}

Let $G$ be an arbitrary group and let $A$ be a trivial $G$-module. In this 
section we introduce a map that links $H^2(G,A)$ to the Schur multiplier of 
a pair of groups. This map will play a central role in our later applications.

\subsection{Twisted cocycles}

For $\delta \in Z^2(G,A)$ we define
\begin{equation}
 \hat{\delta} : G \times G \ra A :
                  (g,h) \ms \delta(g,h) - \delta(h,g)
                       - \delta((hg)^{-1}, hg) 
                       + \delta((hg)^{-1}, gh). 
\end{equation}

For later use we note that $h\in Z(G)$ implies
\begin{equation}
\label{eq_delta}
\hat{\delta}(g,h) = \delta(g,h) - \delta(h,g).
\end{equation}

The following lemma gives an alternative description for $\hat{\delta}$.
Its proof is a direct computation which we omit here. For any two group
elements $g$ and $h$ let $[g,h] = g^{-1} h^{-1} g h$ the commutator of 
$g$ and $h$.

\begin{lemma}
\label{comms}
Let $G$ be a group, $A$ a trivial $G$-module and $\delta \in Z^2(G,
A)$. Then in the extension of $A$ by $G$ via $\delta$ the equation
$[(g,a),(h,b)] = ([g,h],\hat{\delta}(g,h))$ holds for all $g,h \in G$
and $a,b \in A$.
\end{lemma}

\subsection{The Schur multiplier of a pair of groups}

We briefly recall the construction of the non-abelian exterior product
and the non-abelian tensor product as introduced by Brown and Loday 
\cite{BLo87}, see also \cite{BJR87} for details. 

Let $H \leq G$ and let $F$ be the free group on the symbols 
$\{ g \wedge h \mid g \in G, h \in H\}$. For $g, h \in G$ denote
${}^h g := h g h^{-1} = g^{h^{-1}}$. Let $R$ be the normal 
subgroup of $F$ generated by the relations
\begin{eqnarray*}
g g' \wedge h &=& ({}^g g' \wedge {}^g h) (g \wedge h) 
     \;\;\; \mbox{ for } g, g' \in G, h \in H \\
g \wedge h h' &=& (g \wedge h) ({}^h g \wedge {}^h h') 
     \;\;\; \mbox{ for } g \in G, h, h' \in H \\
h \wedge h &=& 1 
     \;\;\; \mbox{ for } h \in H. 
\end{eqnarray*}
Then $G \wedge H := F/R$ is the non-abelian exterior product of $G$
and $H$. If the relations $h \wedge h = 1$ are omitted, then the 
resulting quotient is the non-abelian tensor product of $G$ and $H$.

By construction, there is a natural homomorphism
\[ 
\varphi : G \wedge H \ra [G,H] : g \wedge h \ms ghg^{-1}h^{-1} =
[g^{-1},h^{-1}].
\]
The kernel of $\varphi$ is denoted with $M(G,H)$ and is called the
{\em Schur multiplier} of the pair of groups $(G,H)$, see \cite{Ell98b}
for background. It is known that $M(G,H)$ is an abelian group and 
the (ordinary) Schur multiplier $M(G)$ of the group $G$ can be
obtained as $M(G) = M(G,G)$ by Hopf's formula. Further, if $H \leq Z(G)$, 
then $M(G,H) = G \wedge H$ holds.

\subsection{The linking map}

The following lemma provides the first step for the definition of the
linking map.

\begin{lemma}
\label{wedgehom}
Let $G$ be a group, $A$ a trivial $G$-module, $\delta \in Z^2(G, A)$
and $H \leq Z(G)$. Then the following map is a group homomorphism
\[ \ol{\delta} : M(G,H) \ra A : g \wedge h \ms \hat{\delta}(g,h). \]
\end{lemma}
   
\begin{proof}
Let $E$ be the extension of $A$ by $G$ defined by $\delta$ and define
$\alpha : F \ra E$ as the group homomorphism extending $g \wedge h \ms
([g,h], \hat{\delta}(g,h))$. We note that
$[(g,0),(h,0)] = ([g,h], \hat{\delta}(g,h)) = (1, \hat{\delta}(g,h))$ 
for each $g \in G$ and $h \in H$ by Lemma \ref{comms} and the fact that
$H$ is central in $G$.

We show that $R$ is contained in the kernel of $\alpha$. First, consider
the relation $h \wedge h$ for $h \in H$. As $\hat{\delta}(g,g) = 0$
for all $g\in G$ it follows that $\alpha(h \wedge h) =
([h,h],\hat{\delta}(h,h)) = (1,0)$ in $E$. Next, consider the relation
$g \wedge h h' = (g \wedge h) ({}^h g \wedge {}^h h') = (g \wedge h) 
(g \wedge h')$, as $H$ is central in $G$. Then 
\begin{eqnarray*}
(1, \hat{\delta}(g, hh'))
  &=& [(g, 0), (hh', 0)] \\
  &=& [(g, 0), (h, -\delta(h,h'))(h', 0)] \\
  &=& [(g, 0), (h', 0)] [(g, 0), (h, -\delta(h,h'))]^{(h',0)} \\
  &=& ([g,h'], \hat{\delta}(g,h'))([g,h], \hat{\delta}(g,h))^{(h',0)} \\
  &=& (1, \hat{\delta}(g,h'))(1, \hat{\delta}(g,h))^{(h',0)} \\
  &=& (1, \hat{\delta}(g,h'))(1, \hat{\delta}(g,h)) \\
  &=& (1, \hat{\delta}(g,h) + \hat{\delta}(g,h')).
\end{eqnarray*}
Thus $\alpha( g \wedge h h') = \alpha((g \wedge h) (g \wedge h'))$ as
desired. Finally consider the relation $g g' \wedge h = ({}^g g' \wedge 
{}^g h) (g \wedge h) = ({}^g g' \wedge h)(g \wedge h)$, as $H$ is central
in $G$. Then using a similar computation as above we obtain
\begin{eqnarray*}
(1, \hat{\delta}(gg', h))
  &=& [(gg', 0), (h, 0)] \\
  &=& [(({}^g g') \cdot g, 0), (h, 0)] \\
  &=& (1, \hat{\delta}({}^g g', h) + \hat{\delta}(g,h)) 
\end{eqnarray*}
as desired. In summary, $R \leq ker(\alpha)$ and thus $\alpha$ induces
a group homomorphism $G \wedge H \ra E$ whose image is contained in $A$.
As $G \wedge H = M(G,H)$ for the central subgroup $H$ of $G$, the result 
now follows.
\end{proof}

Lemma \ref{wedgehom} leads to the following definition for the linking
map between $H^2(G, A)$ and the Schur multiplier of a pair of groups.

\begin{theorem}
\label{uct}
Let $G$ be a group, $A$ a trivial $G$-module and $H \leq Z(G)$. Then 
\begin{equation}
\label{eq_kappa}
\kappa_H : Z^2(G, A) \ra \Hom(M(G,H), A): \delta \ms \ol{\delta}
\end{equation}
is a homomorphism of abelian groups with $B^2(G,A) \leq
\ker(\kappa_H)$. Thus it induces
\begin{equation}
\ol{\kappa}_H : H^2(G, A) \ra \Hom(M(G,H), A): [\delta] \ms \ol{\delta}.
\end{equation}
\end{theorem}

\begin{proof}
From Lemma \ref{wedgehom} it follows that $\kappa_H$ is well-defined. The
definition of $\hat{\delta}$  yields that $\kappa_H$ is compatible with the 
addition and inversion in $Z^2(G,A)$ and hence is a group homomorphism.
It remains to show that $B^2(G,A) \leq \ker(\kappa_H)$. Let $\delta \in 
B^2(G,A)$. Then there exists $\epsilon \in C^1(G,A)$ with $\delta(g,h) = 
\epsilon(g) + \epsilon(h) - \epsilon(gh)$ for all $g, h \in G$. Using 
\eqref{eq_delta} it follows for all $g \in G, h\in H$ that 
$\hat{\delta}(g,h) = \epsilon(hg)- \epsilon(gh)=0$. Thus $\delta \in
\ker(\kappa_H)$.
\end{proof}

We note that the universal coefficients theorem for cohomology asserts 
the existence of a short exact sequence
\[ Ext(H_1(G, \Z), A) \emb H^2(G, A) \pro \Hom(M(G), A). \]
If $G$ is abelian, then $\kappa_G$ coincides with the projection in the 
universal coefficients sequence.

\begin{remark}
\label{wedgeact}
Let $H \leq Z(G)$ be a characteristic subgroup of $G$. Then the action of 
$Comp(G,A)$ on $Z^2(G,A)$ is compatible with $\kappa_H$ and thus defines an 
action of $Comp(G,A)$ on $\Hom(M(G,H),A)$.
\end{remark}

For a non-trivial nilpotent group $G$ let $\lambda(G)$ denote its smallest
non-trivial subgroup of the lower central series of $G$. To shorten notation
in our later section we then also denote
\[ \kappa = \kappa_{\lambda(G)} \;\;\; \mbox{ and } \;\;\;
   \ol{\kappa} = \ol{\kappa}_{\lambda(G)}. \]

\section{Lower central series extensions}
\label{lcs}

In this section we describe a construction for a set of isomorphism type 
representatives of lower central series extensions of $A$ by $G$.
If any such extension exists, then $G$ is a nilpotent group and $A$ is
a non-trivial abelian group with a trivial $G$-module structure. We assume 
this throughout the section and we also assume that $G$ is non-trivial to
obtain proper extensions of $A$.  The
following theorem provides a characterisation of the cocycles defining
lower central series extensions.

\begin{theorem}
\label{charlce}
Let $G$ be a non-trivial nilpotent group, $A$ a non-trivial group with 
trivial $G$-module structure and $\delta \in Z^2(G, A)$. Then the extension 
of $A$ by $G$ via $\delta$ is a lower central series extension if and only 
if $\kappa(\delta)$ is surjective.
\end{theorem}

\begin{proof}
Let $E$ denote the extension of $A$ by $G$ via $\delta$ and let $E =
\lambda_1(E) > \lambda_{2}(E) > \dots$ be the lower central series of
$E$. Further let $c$ be the class of $G$.

Suppose that $E$ is a lower central series extension. Then $E$ is nilpotent 
of class $c+1$ and the image $\ol{A}$ of $A$ in $E$ satisfies that $\ol{A} 
= \lambda_{c+1}(E) = [E, \lambda_{c}(E)]$. This implies that $\ol{A} =
\langle [(g,a), (h,b)] \mid a, b \in A, g \in G, h \in \lambda(G) 
\rangle$. As $[(g,a),(h,b)] = (1, \hat{\delta}(g,h))$ by Lemma~\ref{comms} 
and $\hat{\delta}(g,h) = \ol{\delta}(g \wedge h)$ by definition of 
$\ol{\delta}$, it follows that $\ol{\delta} = \kappa(\delta)$ is surjective.

Now suppose that $\kappa(\delta)$ is surjective. Using the same
calculation as in the first part of the proof, it follows that $\ol{A} 
= \lambda_{c+1}(E)$. As $A$ is a trivial $G$-module, this implies that 
$\lambda_{c+2}(E)$ is trivial and thus $E$ is nilpotent. Hence $E$ is a 
lower central series extension of $A$ by $G$.
\end{proof}

The following theorem exhibits an explicit construction for the isomorphism 
types of lower central series extensions of $A$ by $G$. Recall that $Comp(G,A) 
= Aut(G) \times Aut(A)$, as $G$ acts trivially on $A$.

\begin{theorem}
\label{lcsconst}
Let $G$ be a non-trivial nilpotent group and $A$ a non-trivial group with 
trivial $G$-module structure.
Define
\begin{equation}
\label{eq_Delta}
 \Delta = \left\{ \gamma \in H^2(G,A) \mid \ol{\kappa}(\gamma) 
  \mbox{ is surjective} \right\}.
\end{equation}
\begin{items}
\item[\rm (a)]
Then $\Delta$ is invariant under the action of $Comp(G,A)$; denote by
$\Omega$ a set of orbit representatives of the action of $Comp(G,A)$
on $\Delta$.
\item[\rm (b)]
The extensions of $A$ by $G$ defined by the elements in $\Omega$ form a
complete and irredundant set of isomorphism type representatives of lower 
central series extensions of $A$ by $G$.
\item[\rm (c)]
For $\delta$ with $[\delta] \in \Delta$ denote the
corresponding extension by $E$. Then $Aut(E) = Aut_A(E)$ and there
exists a short exact sequence 
\[ Z^1(G,A) \emb Aut(E) \pro Stab_{Comp(G,A)}([\delta]).\]
\end{items}
\end{theorem}

\begin{proof}
Part~(a) follows from Theorem \ref{charlce}. Parts~(b) and~(c) follow 
from Theorem~\ref{sisom}, as lower central series extensions are isomorphic
if and only if they are strongly isomorphic.
\end{proof}

The following remark recalls the structure of $M(G, \lambda(G))$ and
thus gives further insight into the range $Hom(M(G, \lambda(G)), A)$
of $\ol{\kappa}$.

\begin{remark} \label{pairclass}
{\rm (See Prop. 3.2 of \cite{Ell98b})}
Let $G$ be a non-trivial nilpotent group.
\begin{items}
\item[\rm (a)]
If $G$ has class $1$, then $\lambda(G) = G$ and $M(G,\lambda(G)) = M(G) 
\cong G \wedge G$, the abelian exterior product.
\item[\rm (b)]
If $G$ has class at least $2$, then $\lambda(G) \leq G'$
and $M(G,\lambda(G)) \cong G/G' \otimes \lambda(G)$, the abelian tensor 
product.
\end{items}
\end{remark}

\section{Derived series extensions}
\label{ds}

In this section we describe a construction for a set of isomorphism type 
representatives of derived series extensions of $A$ by $G$. Any such
extension is solvable and if one exists, then $G$ is solvable and $A$
is a non-trivial abelian group and a $G$-module. We assume this
throughout the section and further suppose that $G$ is non-trivial to
obtain proper extensions of $A$. Let $\gamma(G)$ denote the smallest
non-trivial subgroup of the derived series of $G$ and set $U = [A,
\gamma(G)]$. We consider the sequence

\begin{equation}
\label{eq_seq}
 Z^2(G, A) \stackrel{\sigma}{\ra} Z^2(\gamma(G), A/U)
             \stackrel{\kappa}{\ra} \Hom( M(\gamma(G)), A/U), 
\end{equation}
where $\sigma(\delta) : \gamma(G) \times \gamma(G) \ra A/U : (g,h) \ms 
\delta(g,h) + U$. Regarding $\kappa$ note that
$\lambda(\gamma(G))=\gamma(G)$ and thus
$M(\gamma(G),\lambda(\gamma(G))) = M(\gamma(G))$. As $\sigma$ maps
$B^2(G,A)$ into $B^2(\gamma(G), A/U)$ equation \eqref{eq_seq} induces
the sequence of maps
\begin{equation}
\label{eq_Hseq}
 H^2(G, A) \stackrel{\ol{\sigma}}{\ra} H^2( \gamma(G), A/U)
             \stackrel{\ol{\kappa}}{\ra} \Hom(M(\gamma(G)), A/U),
\end{equation}
The following theorem characterises the derived series extensions of $G$ 
by $A$.

\begin{theorem}
\label{charder}
Let $G$ be a non-trivial solvable group, $A$ a non-trivial abelian
group with a $G$-module structure and $\delta \in Z^2(G, A)$. 
Then the extension of $A$ by $G$ via $\delta$ is a derived series
extension if and only if $\kappa(\sigma(\delta))$ is surjective.
\end{theorem}

\begin{proof}
Let $E$ denote the extension of $A$ by $G$ via $\delta$. Consider $E
\ra G : (g,a) \ms g$, the natural epimorphism from $E$ onto $G$ and
denote by $M$ the full preimage of $\gamma(G)$ in $E$. By definition
$E$ is a derived series extension if $M' = A$. As $U =
[A, M] \leq M'$, this is equivalent to the condition
$(M/U)' = A/U$. The induced action of $\gamma(G)$ on $A/U$ is
trivial. Thus $M/U$ has class at most $2$ and $\lambda(M/U) =
(M/U)'$. Hence $E$ is a derived series extension if and only if $M/U$
is a lower central series extension of $\gamma(G)$ by $A/U$ via the cocycle
$\sigma(\delta)$. The result now follows from Theorem~\ref{charlce}.
\end{proof}

This allows the following explicit construction for the isomorphism types of
derived series extensions of $A$ by $G$.

\begin{theorem}
\label{derconst}
Let $G$ be a non-trivial solvable group and $A$ a non-trivial group with
an arbitrary $G$-module structure. Define 
\begin{equation}
\label{eq_Delta_der}
\Delta = \left\{ \gamma \in H^2(G,A) \mid \ol{\kappa}(\ol{\sigma}(\gamma)) 
\mbox{ is surjective}\right\}.
\end{equation}
\begin{items}
\item[\rm (a)]
Then $\Delta$ is invariant under the action of $Comp(G,A)$; let $\Omega$
denote a set of orbit representatives of the action of $Comp(G,A)$ on 
$\Delta$.
\item[\rm (b)]
The extensions of $A$ by $G$ defined by the elements in $\Omega$ form a 
complete and irredundant set of isomorphism type representatives of 
derived series extensions of $A$ by $G$.
\item[\rm (c)]
For $\delta$ with $[\delta] \in \Delta$ denote the
corresponding extension by $E$. Then $Aut(E) = Aut_A(E)$ and there
exists a short exact sequence 
\[
Z^1(G,A) \emb Aut(E) \pro Stab_{Comp(G,A)}([\delta]).
\]
\end{items}
\end{theorem}

\begin{proof}
The proof is similar to that of Theorem~\ref{lcsconst}.
\end{proof}

\section{Computational methods}
\label{sec_comp}

In the previous sections we established criteria to decide whether an 
extension is either a lower central series extension or a derived
series extension. Here we exploit these descriptions to obtain
effective algorithms to construct those extensions. We have
implemented the algorithms in GAP~\cite{GAP4} for the special case
that the module is elementary abelian. The implementation is available
in the GAP package SpecialExt~\cite{SpecialExt}.

\subsection{Computing $\mathbf{H^2(G,A)}$ and the action of
  $\mathbf{Comp(G,A)}$}
\label{compute}

If $G$ is a polycyclic group defined by a consistent polycyclic presentation
and $A$ is a finitely generated abelian group and a $G$-module, then 
$H^2(G,A)$ can be computed effectively. We recall the basic ideas of the 
underlying algorithm here briefly for the case that $G$ is finite; see also 
\cite[Section 8.7.2]{HEO05}.

Let $\g = \{g_1, \ldots, g_n\}$ be a polycyclic generating sequence of $G$
and let $\{a_1, \ldots, a_s\}$ be a generating set of $A$. Then there exists 
a (unique) consistent polycyclic presentation of $G$ on the generators
$\g$. This has relations of the form
\[ g_i^{r_i} = g_{i+1}^{e_{i,i+1}} \cdots g_n^{e_{i,n}}
   \mbox{ for } 1 \leq i \leq n, \mbox{ and } \]
\[ g_i^{g_j} = g_{j+1}^{e_{i,j,j+1}} \cdots g_n^{e_{i,j,n}}
   \mbox{ for } 1 \leq j < i \leq n \]
for certain $r_i \in \N$ and $e_{i,j}, e_{i,j,k} \in \Z$. 

Let $l=(n+1)n/2$.
For an extension $E$ of $A$ by $G$ via $\delta \in Z^2(G,A)$ we
define the tuple $t_{\delta}=(t_{i,j}\mid 1\leq j\leq i\leq n) \in A^l$ via 
\[
(1, t_{i,i}) = (g_n,0)^{-e_{i,n}}\cdots (g_{i+1},0)^{-e_{i,i+1}}(g_i,0)^{r_i}
 \mbox{ for } 1 \leq i \leq n, \mbox{ and }
\]
\[
(1, t_{i,j}) = (g_n,0)^{-e_{i,j,n}}\cdots (g_{j+1},0)^{-e_{i,j,j+1}}(g_i,0)^{(g_j,0)}
  \mbox{ for } 1 \leq j < i \leq n.
\]
By \cite[Lemma 8.47]{HEO05} the map
\[ \varphi : Z^2(G,A) \ra A^l : \delta \ms t_{\delta} \]
is a group homomorphism with $\ker(\varphi) \leq B^2(G,A)$. Denote
$Z = \im(\varphi) \leq A^l$ and $B = B^2(G,A)^\varphi \leq Z$. 
Then we obtain an isomorphism
\[ \ol{\varphi} : H^2(G,A) \ra Z/B. \]

For $t\in A^l$ we define a presentation $P(t)$ with generators $g_1,
\ldots, g_n, a_1, \ldots, a_s$ and two types of relations: firstly
relations defining $A$ as group and $G$-module and secondly 
 \[ g_i^{r_i} = g_{i+1}^{e_{i,i+1}} \cdots g_n^{e_{i,n}} \cdot t_{i,i}
    \mbox{ for } 1 \leq i \leq n, \]
 \[ g_i^{g_j} = g_{j+1}^{e_{i,j,j+1}} \cdots g_n^{e_{i,j,n}} \cdot t_{i,j}
    \mbox{ for } 1 \leq j < i \leq n.\]
If $t\in Z$, then the presentation $P(t)$ 
defines a group that is an extension of $A$ by $G$ via a cocycle
$\delta$ in the preimage of $t$ under $\varphi$.

The subgroups $Z$ and $B$ of $A^l$ can be computed effectively if $A$
is elementary-abelian (see~\cite[Section 8.7.2]{HEO05}) and we
use $Z/B$ to represent $H^2(G,A)$ in our applications. 
The group $Comp(G,A)$ can be computed via its definition
The action of $Comp(G,A)$ on $H^2(G,A)$
translates to an action of $Comp(G,A)$ on $Z/B$. A pair
$(\eta,\nu)\in Comp(G,A) \leq Aut(G)\times Aut(A)$ acts on $t\in
Z$ via
\[
(1,t_{i,i})^{(\eta,\nu)} =
\left((g_n^{\eta^{-1}},0)^{-e_{i,n}}\cdots
  (g_{i+1}^{\eta^{-1}},0)^{-e_{i,i+1}}(g_i^{\eta^{-1}},0)^{r_i}\right)^{\nu}
 \mbox{ for } 1 \leq i \leq n, \mbox{ and }
\]
\[
(1, t_{i,j})^{(\eta,\nu)} = \left(
(g_n^{\eta^{-1}},0)^{-e_{i,j,n}}\cdots (g_{j+1}^{\eta^{-1}},0)^{-e_{i,j,j+1}}(g_i^{\eta^{-1}},0)^{(g_j^{\eta^{-1}},0)}\right)^{\nu}
  \mbox{ for } 1 \leq j < i \leq n.
\]

\subsection{Lower central series extensions}
\label{complcs}

Let $G$ be a finite non-trivial nilpotent group and $A$ a finite
non-trivial abelian group with a trivial $G$-module structure. Our
aim is to compute a complete and irredundant set of isomorphism type
representatives of lower central series extensions of $A$ by $G$.

We choose a polycyclic generating sequence $\g = \{ g_1, \ldots, g_n \}$
such that it refines the lower central series of $G$. As $G'=[G,G]$
and $\lambda(G)$ are
subgroups in the lower central series of $G$, it then follows that there
exist indices $d$ and $m$ in $\{1, \ldots, n\}$ with $G' = \langle g_{d+1}, 
\ldots, g_n \rangle$  and $\lambda(G) = \langle g_m, \ldots, g_n \rangle$. Let 
$J = \{ (i,j) \mid 1 \leq j\leq i \leq n \}$ and $I = \{ (i,j) \mid 1 \leq j 
\leq d, m \leq i \leq n, j<i\} \subseteq J$. Then $|J| = l$ and we denote
$|I| = h$. We further denote the natural projection corresponding to
$I$ and $J$ by
\[ \pi : A^l \ra A^h 
       : (t_{i,j} \mid (i,j) \in J) \ms (t_{i,j} \mid (i,j) \in I).\]

By Theorem~\ref{lcsconst} we have achieved our aim if we find
$\varphi(\Omega)$ for a set $\Omega$ as defined in
Theorem~\ref{lcsconst}. Recall that $\Omega$ is a set of orbit
representatives in the set $\Delta$ from \eqref{eq_Delta}. The
following lemma provides a criterion to check whether a given $\delta
\in Z^2(G,A)$ satisfies $[\delta] \in \Delta$. We say that a tuple
$t \in A^k$ for $k\in \N$ is \emph{full} if the entries in $t$ generate $A$.

\begin{lemma}
\label{full}
Let $G$ be a non-trivial nilpotent group, $A$ a non-trivial abelian
group with  trivial $G$-module structure and $\delta \in Z^2(G, A)$. 
Then $\kappa(\delta)$ is surjective if and
only if $\pi(t_{\delta})$ is full.
\end{lemma}

\begin{proof}
Recall that $\kappa(\delta) : M(G,\lambda(G)) \ra A$ with
$\kappa(\delta)(g \wedge h) = \hat{\delta}(g,h) = \delta(g,h) -
\delta(h,g)$. Thus $\kappa(\delta)$ is surjective if and only if $\{
\hat{\delta}(g,h) \mid g \in G, h \in \lambda(G) \} = A$. As
$\lambda(G)$ is central in $G$, the definition of $M(G,\lambda(G))$
and Lemma  
\ref{wedgehom} yield that $\hat{\delta}(g, hh') = \hat{\delta}(g,h) +
\hat{\delta}(g,h')$ and $\hat{\delta}(gg',h) = \hat{\delta}(^gg', h) +
\hat{\delta}(g,h)$. By Remark \ref{pairclass} we note that the first
argument of $\hat{\delta}$ depends on $G/G'$ only. This also yields that
$\hat{\delta}$ is multiplicative in both components. In summary, it
follows that $\kappa(\delta)$ is surjective if and only if 
$\langle \hat{\delta}(g_j, g_i) \mid (i,j) \in I \rangle = A$ and thus
if and only if $\pi(t_{\delta})$ is full.
\end{proof}

Following the approach described in Section~\ref{compute} we first
compute $Z/B$ and then its orbits under the action of
$Comp(G,A)$. We next choose orbit representatives $r_1,\dots, r_m\in
Z/B$ and then a representative $t_i$ in $A^l$ for each $r_i, 1\leq i\leq
m$. According to Theorem~\ref{lcsconst}(a) and Lemma~\ref{full} we
then have $\varphi(\Omega) = \{t_i \mid \pi(t_i) \mbox{ is full}\}$ for
some $\Omega$ as desired and $\{P(t) \mid t\in \varphi(\Omega)\}$ is a
set of presentations for a complete and irredundant list of
lower central series extensions of $A$ by $G$.

To avoid redundant computations it is useful to know {\em a priori} 
when there are no lower central series extensions of $A$ by $G$. The
following remark collects some elementary conditions for this purpose.

\begin{remark}
There are no lower central series  extensions of $A$ by $G$ if either
$h < d(A)$, where $d(A)$ is the minimal generator number of $A$, or
$\pi(\im(\varphi)) \leq M^h$ for some proper subgroup $M$ of $A$.
\end{remark}

A computationally more involved criterium enables us to always detect
when no lower central series extension exists. Denote the maximal
subgroups of $A$ by $M_1, \ldots, M_s$ and let $Z_i := \{ t \in Z \mid 
\pi(t) \in M_i^h \}$. Then 
\[ \varphi(\Delta) = Z \setminus \cup_{i=1}^s Z_i. \]
We use the Inclusion-Exclusion Principle to determine the cardinality
of $\varphi(\Delta)$ via
\begin{eqnarray*}
|\varphi(\Delta)| 
  &=& |Z| - | \cup_{i = 1}^s Z_i | \\ 
  &=& |Z| - \left(\sum_{k=1}^s (-1)^{k+1} 
    (\sum_{1 \leq i_1 < \ldots < i_k \leq s} 
     |Z_{i_1} \cap \ldots \cap Z_{i_k}|)\right)
\end{eqnarray*}
Each intersection $Z_{i_1} \cap \ldots \cap Z_{i_k}$ can
be computed readily from $M_{i_1},\dots,M_{i_k}$ and $Z$ by solving
a system of linear equations. In the special case that $A$ is
elementary abelian this approach simplifies as the automorphism group
acts as full symmetric group on the maximal subgroups and only one
intersection of every type has to be determined.

\subsection{Derived series extensions}
\label{compds}

Let $G$ be a finite non-trivial solvable group and $A$ a finite
non-trivial abelian group with a $G$-module structure. Our aim is to
compute a complete and irredundant set of isomorphism type
representatives of derived series extensions of $A$ by $G$.

We choose a polycyclic generating sequence $\g = \{ g_1, \ldots, g_n \}$
such that it refines the derived series of $G$. As $\gamma(G)$ is a
subgroup of this series, it follows that there exists $m \in \{1, \ldots,
n\}$ with $\gamma(G) = \langle g_m, \ldots, g_n \rangle$. Let $J = \{
(i,j) \mid 1 \leq j \leq i 
\leq n \}$ and $I = \{ (i,j) \mid m \leq j < i \leq n \}$. Then $|J| = l$ 
and we denote $|I| = k$. We denote the 
natural projection corresponding to $I$, $J$ and $A/[A,\gamma(G)]$ by 
\[ \mu : A^l \ra (A/[A,\gamma(G)])^k 
       : (t_{i,j} \mid (i,j) \in J) \ms (t_{i,j} + [A,\gamma(G)] \mid (i,j) \in I).\]

By Theorem~\ref{derconst} we have achieved our aim if we find
$\varphi(\Omega)$ for a set $\Omega$ as defined in
Theorem~\ref{derconst}. Recall that $\Omega$ is a set of orbit
representatives in the set $\Delta$ from \eqref{eq_Delta_der}. The
following lemma provides a criterion to check whether a given $\delta
\in Z^2(G,A)$ satisfies $[\delta] \in \Delta$. 

\begin{lemma}
\label{full2}
Let $G$ be a non-trivial solvable group, $A$ a non-trivial abelian
group with a $G$-module structure and $\delta \in Z^2(G, A)$.
Then $\kappa(\sigma(\delta))$ is surjective if and only if 
$\mu(t_{\delta})$ is full.
\end{lemma}

\begin{proof}
The induced cocycle $\sigma(\delta)$ defines an extension of
$A/[A,\gamma(G)]$ by $\gamma(G)$. Applying Lemma~\ref{full} yields
that $\kappa(\sigma(\delta))$ is surjective if and only if
$\pi(t_{\sigma(\delta)})$ is full. The latter equals $\mu(t_{\delta})$
which completes the proof.
\end{proof}

Following the approach described in Section \ref{compute} we first
compute $Z/B$ and then its orbits under the action of $Comp(G,A)$. 
 We next choose orbit representatives $r_1,\dots, r_m\in
Z/B$ and then a representative $t_i$ in $A^l$ for each $r_i, 1\leq i\leq
m$. According to Theorem~\ref{derconst}(a) and Lemma~\ref{full2} we
then have $\varphi(\Omega) = \{t_i \mid \mu(t_i) \mbox{ is full}\}$ for
some $\Omega$ as desired and $\{P(t) \mid t\in \varphi(\Omega)\}$ is a
set of presentations for a complete and irredundant list of
derived series extensions of $A$ by $G$.

\section{Application}
\label{sec_dl}

Groups of given derived length and minimal composition length are
known up to derived length 8, see~\cite{Gla05}. Here we use the method from
Section~\ref{compds} to obtain new information on groups of derived
length 10. For our computations we utilised the GAP package
SpecialExt~\cite{SpecialExt} which implements the methods from
Section~\ref{sec_comp} for elementary abelian groups.

The sporadic simple group $Fi_{23}$ has a maximal solvable subgroup $M$ of 
order $2^{11} 3^{13}$. This group has the form 
$M= \operatorname{GL}(2,3) \ltimes 3^{2+1} \ltimes 2^{6+1} \ltimes 3^{8+1}$,
where $p^{r+s}$ is used to denote an $r$-generator group of order $p^{r+s}$ 
and class $2$. The group $M$ has derived length $10$ and composition length
$24$. It is conjectured that $24$ is the minimal possible composition length
for a group of derived length $10$, see \cite{Gla05}. Previously the
group $M$ has been the only group known to achieve this bound.

Let $M = M^{(1)} \geq M^{(2)} \geq \ldots \geq M^{(11)} = \{1\}$ denote
the derived series of $M$. We have used the method from
Section~\ref{compds} to determine a complete and
irredundant set of isomorphism type representatives of derived series 
extensions of $A := M^{(i)}/M^{(i+1)}$ by $G := M/M^{(i)}$ for $2\leq
i \leq 10$ using the
$G$-module structure of $A$ inside $M$. The following four values of
$i$ yield more than one isomorphism type representative.

\begin{description}
\item[$\mathbf{i=4:}$] Here $G\cong S_4$ and $A\cong C_2$ is a trivial
  $G$-module. We obtain two non-isomorphic derived series extensions:
  the groups with the numbers $28$ and $29$ in the SmallGroups
  Library described in~\cite{BEO02}.
\item [$\mathbf{i=6:}$] Here $G\cong \GL_2(3) \ltimes 3^2$ and $A\cong 
  C_3^2$. We obtain three non-isomorphic derived series extensions: the
  groups with the numbers $2889$, $2890$ and $2891$ in the SmallGroups
  Library.
\item [$\mathbf{i=8:}$] Here $G \cong \GL_2(3) \ltimes 3^{2+1}
  \ltimes 2^6$ and $A \cong C_2$ is a trivial $G$-module. We obtain
  two non-isomorphic derived series extensions. 
\item [$\mathbf{i=10:}$] Here $G \cong \GL_2(3) \ltimes 3^{2+1}
  \ltimes 2^{6+1} \ltimes 3^8$ and $A \cong C_3$. We obtain three
  non-isomorphic derived series extensions. 
\end{description}

In particular, the computation for $i$ equal to $10$ yields two {\it
  new} examples of 
groups of derived length $10$ and composition length $24$. We describe
all three groups arising from this case via a parametrised polycyclic
presentation. This has $24$ generators $g_1, \ldots,
g_{24}$ and the relations exhibited in Figures 1 and 2 where
conjugation relations of the form $g_i^{g_j} = g_i$ are omitted in the
latter. The relations with left hand sides $g_2^3$ and $g_2^{g_1}$
contain a parameter $k$ and the three different groups are obtained
for $k \in \{0, 1, 2\}$. The presentations are also available as examples
in the GAP package SpecialExt. Using GAP it is straightforward to
determine the orders of conjugacy class representatives of each of
these three groups, which yields an independent check for the
non-isomorphism of the groups.

\begin{figure}[htb]
\begin{center}
{\small
\begin{align*}
& 
g_{1}^2=1,
g_{2}^3=g_{24}^k,
g_{3}^2=g_{5},
g_{4}^2=g_{5},
g_{5}^2=1,
g_{6}^3=1,
g_{7}^3=1,
g_{8}^3=1,
g_{9}^2=g_{15},
g_{10}^2=g_{15},
g_{11}^2=g_{15},\\
&
g_{12}^2=g_{15},  
g_{13}^2=g_{15},
g_{14}^2=g_{15},  
g_{15}^2=1,
g_{16}^3=1, 
g_{17}^3=1, 
g_{18}^3=1, 
g_{19}^3=1,  
g_{20}^3=1,
g_{21}^3=1,\\
& 
g_{22}^3=1, 
g_{23}^3=1,  
g_{24}^3=1
\end{align*} }
\end{center}
\caption{Power relations}
\end{figure}

\begin{figure}[htb]
\begin{center}
{\small
\begin{align*}
&
g_{2}^{g_{1}}=g_{2}^{2}g_{24}^{3-k},
g_{3}^{g_{1}}=g_{3}g_{5},
g_{3}^{g_{2}}=g_{3}g_{4},
g_{4}^{g_{1}}=g_{3}g_{4}g_{5},
g_{4}^{g_{2}}=g_{3},
g_{4}^{g_{3}}=g_{4}g_{5},
g_{6}^{g_{3}}=g_{6}^{2}g_{7}g_{8},
g_{6}^{g_{4}}=g_{6}g_{7}g_{8}^{2},\\
&
g_{6}^{g_{5}}=g_{6}^{2},
g_{7}^{g_{1}}=g_{6}^{2}g_{7}^{2}g_{8}^{2},
g_{7}^{g_{2}}=g_{6}g_{7}g_{8}^{2},
g_{7}^{g_{3}}=g_{6}g_{7}g_{8}^{2},
g_{7}^{g_{4}}=g_{6}g_{7}^{2}g_{8},
g_{7}^{g_{5}}=g_{7}^{2},
g_{7}^{g_{6}}=g_{7}g_{8},
g_{8}^{g_{1}}=g_{8}^{2},\\
&
g_{9}^{g_{1}}=g_{10}g_{11}g_{12}g_{13}g_{14}g_{15},
g_{9}^{g_{2}}=g_{9}g_{12}g_{14},
g_{9}^{g_{3}}=g_{9}g_{12}g_{14},
g_{9}^{g_{4}}=g_{9}g_{11}g_{12}g_{13}g_{14},
g_{9}^{g_{6}}=g_{13}, 
g_{9}^{g_{8}}=g_{10},\\
&
g_{10}^{g_{1}}=g_{9}g_{12}g_{14}g_{15}, 
g_{10}^{g_{2}}=g_{10}g_{11}g_{12}g_{13}g_{14},
g_{10}^{g_{3}}=g_{10}g_{11}g_{12}g_{13}g_{14}, 
g_{10}^{g_{4}}=g_{10}g_{11}g_{13},
g_{10}^{g_{6}}=g_{14},
\\
& 
g_{10}^{g_{8}}=g_{9}g_{10}, 
g_{10}^{g_{9}}=g_{10}g_{15},
g_{11}^{g_{1}}=g_{9}g_{10}g_{12}g_{13}g_{14}g_{15}, 
g_{11}^{g_{2}}=g_{10}g_{11}g_{14}, 
g_{11}^{g_{3}}=g_{9}g_{10}g_{12}g_{13}g_{14},\\
&
g_{11}^{g_{4}}=g_{10}g_{12}g_{13}g_{14},
g_{11}^{g_{5}}=g_{13},
g_{11}^{g_{6}}=g_{9},
g_{11}^{g_{7}}=g_{12},
g_{11}^{g_{8}}=g_{12},
g_{12}^{g_{1}}=g_{10}g_{11}g_{14}g_{15},\\
&
g_{12}^{g_{2}}=g_{9}g_{10}g_{12}g_{13}g_{14},
g_{12}^{g_{3}}=g_{9}g_{11}g_{12}g_{13},
g_{12}^{g_{4}}=g_{9}g_{10}g_{11}g_{12}g_{13},
g_{12}^{g_{5}}=g_{14},
g_{12}^{g_{6}}=g_{10},
g_{12}^{g_{7}}=g_{11}g_{12},\\
&
g_{12}^{g_{8}}=g_{11}g_{12},
g_{12}^{g_{11}}=g_{12}g_{15},
g_{13}^{g_{1}}=g_{9}g_{10}g_{11}g_{12}g_{14}g_{15},
g_{13}^{g_{2}}=g_{10}g_{12}g_{13},
g_{13}^{g_{3}}=g_{9}g_{10}g_{11}g_{12}g_{14},\\
&
g_{13}^{g_{4}}=g_{10}g_{11}g_{12}g_{14},
g_{13}^{g_{5}}=g_{11},
g_{13}^{g_{6}}=g_{11},
g_{13}^{g_{7}}=g_{13}g_{14},
g_{13}^{g_{8}}=g_{14},
g_{14}^{g_{1}}=g_{10}g_{12}g_{13}g_{15},\\
&
g_{14}^{g_{2}}=g_{9}g_{10}g_{11}g_{12}g_{14},
g_{14}^{g_{3}}=g_{9}g_{11}g_{13}g_{14},
g_{14}^{g_{4}}=g_{9}g_{10}g_{11}g_{13}g_{14},
g_{14}^{g_{5}}=g_{12},
g_{14}^{g_{6}}=g_{12},
g_{14}^{g_{7}}=g_{13},\\
&
g_{14}^{g_{8}}=g_{13}g_{14}, 
g_{14}^{g_{13}}=g_{14}g_{15}, 
g_{16}^{g_{1}}=g_{16}^{2}g_{17}^{2}g_{18}^{2}g_{20}^{2}g_{24}, 
g_{16}^{g_{2}}=g_{16}g_{17}g_{18}g_{20}g_{24}^2,
g_{16}^{g_{3}}=g_{16}g_{17}g_{18}g_{20}g_{24}^2,\\ 
&
g_{16}^{g_{4}}=g_{16}^{2}g_{17}g_{18}g_{20}, 
g_{16}^{g_{5}}=g_{16}^{2}g_{24}, 
g_{16}^{g_{9}}=g_{16}^{2}g_{20}, 
g_{16}^{g_{10}}=g_{16}g_{20}g_{24}^2,
g_{16}^{g_{11}}=g_{16}^{2}g_{18}, 
g_{16}^{g_{12}}=g_{16}g_{18}g_{24}^2,\\ 
&
g_{16}^{g_{13}}=g_{16}^{2}g_{17}g_{24}^2,
g_{16}^{g_{14}}=g_{16}g_{17}g_{24}, 
g_{16}^{g_{15}}=g_{16}^{2}g_{24},
g_{17}^{g_{1}}=g_{16}^{2}g_{17}^{2}g_{18}g_{19}^{2}g_{20}g_{21}^{2}g_{22}^{2},\\ 
&
g_{17}^{g_{2}}=g_{16}^{2}g_{17}^{2}g_{18}g_{19}^{2}g_{20}g_{21}^{2}g_{22}^{2}g_{24},
g_{17}^{g_{3}}=g_{16}g_{17}g_{19}^{2}g_{21}^{2}g_{22}^{2}, 
g_{17}^{g_{4}}=g_{16}g_{18}^{2}g_{19}^{2}g_{21}^{2}g_{22}^{2}g_{24},
g_{17}^{g_{5}}=g_{18}^{2},\\
&  
g_{17}^{g_{6}}=g_{18}g_{24}, 
g_{17}^{g_{7}}=g_{16}g_{17}g_{24}^2,
g_{17}^{g_{8}}=g_{16}^{2}g_{17}g_{24}^2, 
g_{17}^{g_{9}}=g_{17}^{2}g_{21}, 
g_{17}^{g_{10}}=g_{17}g_{21}g_{24}^2, 
g_{17}^{g_{11}}=g_{17}^{2}g_{19},\\ 
& 
g_{17}^{g_{12}}=g_{17}g_{19}g_{24}^2,
g_{17}^{g_{13}}=g_{16}g_{17}g_{24}, 
g_{17}^{g_{14}}=g_{16}g_{17}^{2}g_{24}^2, 
g_{17}^{g_{15}}=g_{17}^{2}g_{24}, 
g_{18}^{g_{1}}=g_{16}^{2}g_{17}g_{18}^{2}g_{19}^{2}g_{20}g_{21}^{2}g_{22}^{2}g_{24}^2,\\
&
g_{18}^{g_{2}}=g_{16}^{2}g_{17}g_{18}^{2}g_{19}^{2}g_{20}g_{21}^{2}g_{22}^{2}g_{24}, 
g_{18}^{g_{3}}=g_{16}g_{18}g_{19}^{2}g_{21}^{2}g_{22}^{2}, 
g_{18}^{g_{4}}=g_{16}g_{17}^{2}g_{19}^{2}g_{21}^{2}g_{22}^{2}g_{24},
g_{18}^{g_{5}}=g_{17}^{2},\\
& 
g_{18}^{g_{6}}=g_{20},
g_{18}^{g_{7}}=g_{16}^{2}g_{18}g_{24}^2,
g_{18}^{g_{8}}=g_{16}^{2}g_{18}g_{24}^2, 
g_{18}^{g_{9}}=g_{18}^{2}g_{22},
g_{18}^{g_{10}}=g_{18}g_{22}g_{24},
g_{18}^{g_{11}}=g_{16}g_{18}g_{24},\\
&
g_{18}^{g_{12}}=g_{16}g_{18}^{2},   
g_{18}^{g_{13}}=g_{18}^{2}g_{19}g_{24}, 
g_{18}^{g_{14}}=g_{18}g_{19}g_{24}^2, 
g_{18}^{g_{15}}=g_{18}^{2}g_{24}^2,
g_{19}^{g_{1}}=g_{17}^{2}g_{18}^{2}g_{19}g_{20}^{2}g_{21}^{2}g_{22}^{2}g_{23}^{2},\\
& 
g_{19}^{g_{2}}=g_{17}g_{18}g_{19}^{2}g_{20}g_{21}g_{22}g_{23},
g_{19}^{g_{3}}=g_{17}^{2}g_{18}^{2}g_{19}g_{23}g_{24}, 
g_{19}^{g_{4}}=g_{17}^{2}g_{18}^{2}g_{19}^{2}g_{23}g_{24}, 
g_{19}^{g_{5}}=g_{19}^{2}g_{24}^2,\\
&  
g_{19}^{g_{6}}=g_{22}g_{24}^2,
g_{19}^{g_{7}}=g_{16}^{2}g_{17}^{2}g_{18}g_{19}, 
g_{19}^{g_{8}}=g_{16}g_{17}^{2}g_{18}^{2}g_{19}g_{24}, 
g_{19}^{g_{9}}=g_{19}^{2}g_{23}g_{24}, 
g_{19}^{g_{10}}=g_{19}g_{23}g_{24}^2,\\
&  
g_{19}^{g_{11}}=g_{17}g_{19}g_{24},
g_{19}^{g_{12}}=g_{17}g_{19}^{2}, 
g_{19}^{g_{13}}=g_{18}g_{19}g_{24}^2, 
g_{19}^{g_{14}}=g_{18}g_{19}^{2}g_{24}, 
g_{19}^{g_{15}}=g_{19}^{2}g_{24}^2, \\
&
g_{20}^{g_{1}}=g_{16}^{2}g_{17}g_{18}g_{19}^{2}g_{20}^{2}g_{21}^{2}g_{22}^{2}, 
g_{20}^{g_{2}}=g_{16}^{2}g_{17}g_{18}g_{19}^{2}g_{20}^{2}g_{21}^{2}g_{22}^{2}g_{24}^2,
g_{20}^{g_{3}}=g_{16}^{2}g_{17}g_{18}g_{19}^{2}g_{20}^{2}g_{21}^{2}g_{22}^{2}g_{24}^2,\\
&
g_{20}^{g_{4}}=g_{16}^{2}g_{17}^{2}g_{18}^{2}g_{19}^{2}g_{20}g_{21}^{2}g_{22}^{2}, 
g_{20}^{g_{5}}=g_{20}^{2}g_{24}^2, 
g_{20}^{g_{6}}=g_{17}g_{24}^2, 
g_{20}^{g_{8}}=g_{16}^{2}g_{20}g_{24}^2, 
g_{20}^{g_{9}}=g_{16}g_{20}g_{24},\\
& 
g_{20}^{g_{10}}=g_{16}g_{20}^{2},
g_{20}^{g_{11}}=g_{20}^{2}g_{22},  
g_{20}^{g_{12}}=g_{20}g_{22}g_{24}, 
g_{20}^{g_{13}}=g_{20}^{2}g_{21}g_{24}, 
g_{20}^{g_{14}}=g_{20}g_{21}g_{24}^2, 
g_{20}^{g_{15}}=g_{20}^{2}g_{24}^2,\\
& 
g_{20}^{g_{19}}=g_{20}g_{24}^2, 
g_{21}^{g_{1}}=g_{17}^{2}g_{18}^{2}g_{19}^{2}g_{20}^{2}g_{21}g_{22}^{2}g_{23}^{2}g_{24},
g_{21}^{g_{2}}=g_{17}g_{18}g_{19}g_{20}g_{21}^{2}g_{22}g_{23}g_{24},\\
& 
g_{21}^{g_{3}}=g_{16}g_{19}^{2}g_{20}^{2}g_{22}^{2}g_{23}g_{24},
g_{21}^{g_{4}}=g_{16}^{2}g_{19}g_{20}^{2}g_{21}g_{23}g_{24}^{2},
g_{21}^{g_{5}}=g_{22}^{2}, 
g_{21}^{g_{6}}=g_{19}, 
g_{21}^{g_{7}}=g_{20}g_{21}g_{24}^2,\\
&
g_{21}^{g_{8}}=g_{16}g_{17}^{2}g_{20}^{2}g_{21}g_{24}, 
g_{21}^{g_{9}}=g_{17}g_{21}g_{24},
g_{21}^{g_{10}}=g_{17}g_{21}^{2},
g_{21}^{g_{11}}=g_{21}^{2}g_{23}g_{24}, 
g_{21}^{g_{12}}=g_{21}g_{23}g_{24}^2,\\
&
g_{21}^{g_{13}}=g_{20}g_{21}g_{24}^2, 
g_{21}^{g_{14}}=g_{20}g_{21}^{2}g_{24},
g_{21}^{g_{15}}=g_{21}^{2}g_{24}^2, 
g_{21}^{g_{18}}=g_{21}g_{24},
g_{22}^{g_{1}}=g_{17}^{2}g_{18}^{2}g_{19}^{2}g_{20}^{2}g_{21}^{2}g_{22}g_{23}^{2}g_{24},\\
&  
g_{22}^{g_{2}}=g_{17}g_{18}g_{19}g_{20}g_{21}g_{22}^{2}g_{23}g_{24},
g_{22}^{g_{3}}=g_{16}g_{19}^{2}g_{20}^{2}g_{21}^{2}g_{23}g_{24},
g_{22}^{g_{4}}=g_{16}^{2}g_{19}g_{20}^{2}g_{22}g_{23}g_{24}^{2}, 
g_{22}^{g_{5}}=g_{21}^{2},\\
&   
g_{22}^{g_{6}}=g_{21}g_{24}, 
g_{22}^{g_{7}}=g_{20}^{2}g_{22}g_{24},
g_{22}^{g_{8}}=g_{16}g_{18}^{2}g_{20}^{2}g_{22}, 
g_{22}^{g_{9}}=g_{18}g_{22}g_{24}^2, 
g_{22}^{g_{10}}=g_{18}g_{22}^{2}, 
g_{22}^{g_{11}}=g_{20}g_{22}g_{24}^2,\\
&  
g_{22}^{g_{12}}=g_{20}g_{22}^{2},
g_{22}^{g_{13}}=g_{22}^{2}g_{23},
g_{22}^{g_{14}}=g_{22}g_{23}g_{24}^2, 
g_{22}^{g_{15}}=g_{22}^{2}g_{24}, 
g_{22}^{g_{17}}=g_{22}g_{24}, 
g_{23}^{g_{1}}=g_{19}^{2}g_{21}^{2}g_{22}^{2}g_{23}g_{24}^2,\\
&  
g_{23}^{g_{2}}=g_{19}^{2}g_{21}^{2}g_{22}^{2}g_{23}g_{24},
g_{23}^{g_{3}}=g_{16}^{2}g_{17}^{2}g_{18}^{2}g_{19}^{2}g_{21}g_{22}g_{23}^{2}g_{24}^2, 
g_{23}^{g_{4}}=g_{16}^{2}g_{17}g_{18}g_{19}^{2}g_{21}g_{22}g_{23}, 
g_{23}^{g_{5}}=g_{23}^{2}g_{24}^2,\\
& 
g_{23}^{g_{7}}=g_{20}^{2}g_{21}^{2}g_{22}g_{23}, 
g_{23}^{g_{8}}=g_{16}^{2}g_{17}g_{18}g_{19}^{2}g_{20}g_{21}^{2}g_{22}^{2}g_{23}g_{24}^2, 
g_{23}^{g_{9}}=g_{19}g_{23}g_{24}^2,
g_{23}^{g_{10}}=g_{19}g_{23}^{2}g_{24},\\ 
&   
g_{23}^{g_{11}}=g_{21}g_{23}g_{24}^2,
g_{23}^{g_{12}}=g_{21}g_{23}^{2}g_{24}, 
g_{23}^{g_{13}}=g_{22}g_{23}g_{24}, 
g_{23}^{g_{14}}=g_{22}g_{23}^{2}, 
g_{23}^{g_{15}}=g_{23}^{2}g_{24}^2,
g_{23}^{g_{16}}=g_{23}g_{24}^2,\\
&    
g_{24}^{g_{1}}=g_{24}^{2}
\end{align*} }
\end{center}
\caption{Conjugation relations}
\end{figure}

\section*{Acknowledgements}

We thank Stephen Glasby for pointing us at the interesting example of
the group $M$ from Section~\ref{sec_dl}.

\def\cprime{$'$}

\end{document}